\documentclass[11pt]{amsart} 

\usepackage{setspace} \onehalfspace
\usepackage[active]{srcltx}
\usepackage[all]{xy} 
\usepackage{amsmath} 
\usepackage{amssymb}
\usepackage{color}
\usepackage{xcolor}
\usepackage{amsthm} \usepackage[pdfauthor={Matei Toma},
pdftitle={Positive currents on non-k\"ahlerian surfaces},
pdfkeywords={Douady space, coherent sheaves, complex space}, pdfview={Fit}]{hyperref}

\AtBeginDocument{%
   \def\MR#1{}
} 

\newtheorem{Th}{Theorem}[section] 

\newtheorem{Lemma}[Th]{Lemma}
\newtheorem{Cor}[Th]{Corollary} 
\newtheorem{Prop}[Th]{Proposition}
\newtheorem{Conj}[Th]{Conjecture} 

\newtheorem{Def}[Th]{Definition}

\newtheorem{Rem}[Th]{Remark}

\def\cF{\mathcal F}

\newcommand{\C}{\mathbb{C}}

\newcommand{\N}{\mathbb{N}}

\newcommand{\R}{\mathbb{R}}
\newcommand{\T}{\mathbb{T}}

\newcommand{\Z}{\mathbb{Z}}

\def\Arg{\mbox{Arg}}

\def\d{\mbox{d}}

\def\Ker{\mbox{Ker}}

\def\Supp{\mbox{Supp}}

\usepackage{combelow}
\title{Positive currents on non-k\"ahlerian surfaces, II}

\author{Ionu\cb{t} Chiose and Matei Toma}

\address{Ionu\cb{t} Chiose, 
Institute of Mathematics of the Romanian Academy,
Bucharest, Romania}
\email{Ionut.Chiose@imar.ro}

\address{Matei Toma, 
Universit\'e de Lorraine, CNRS, IECL, F-54000 Nancy, France}
\email{Matei.Toma@univ-lorraine.fr}

 \date{\today}
\thanks{ AMS
  Classification (2020): 32J15; secondary: 32U40.}

\begin{document}

\dedicatory{To the memory of Mihnea Col\cb{t}oiu}

\begin{abstract}
We present some properties of positive closed currents of type $(1,1)$ on compact non-k\"ahlerian surfaces related to our previous study of these objects started in \cite{ChiTo2}.
\end{abstract}
\maketitle

\noindent

\section{Introduction}

The first systematic study of compact non-k\"ahlerian complex  surfaces was done by Kodaira as part of his comprehensive work on the classification of compact complex surfaces in the sixties. This was followed by work of M. Inoue and Ma. Kato in the seventies who constructed examples of new classes of non-k\"ahlerian  surfaces. Together with the elliptic surfaces and the Hopf surfaces these are the only classes of compact non-k\"ahlerian  surfaces known to this day.  In fact  the Global Spherical Shell Conjecture 
 claims that any non-k\"ahlerian compact complex surface should belong to one of these  classes, \cite{NakamuraSugaku}, see also  \cite{TelemanAnnals}.  A positive answer to this conjecture would mean finishing the last step lacking for a complete classification of all compact complex surfaces up to bimeromorphic equivalence. 

One approach to the study of compact non-k\"ahlerian surfaces  has been through the algebraic objects they may admit, such as algebraic curves. It is a problem though  to show that such objects exist at all on compact non-k\"ahlerian surfaces. Instead,  we know by a result of Harvey and Lawson that the lack of K\"ahler metrics is equivalent to the existence of positive currents which are $(1,1)$-components of a boundary, \cite{HarveyLawson}. 
This result was further refined in \cite{LamariJMPA}  where it is shown  that any compact non-k\"ahlerian surface admits non-trivial positive $\d$-exact currents of type $(1,1)$. In \cite{HarveyLawson} the structure of these currents was described for elliptic, Hopf and Inoue surfaces. 
This description was extended to the other class of known surfaces, that of Kato surfaces, in \cite{TomaCurrents}. Inspired by these structural results, by phenomena appearing in the study of the K\"ahler rank on surfaces, \cite{ChiTo}, as well as by Brunella's postumuous paper \cite{BrunellaGreen2}, we introduce in \cite{ChiTo2} an invariant on a class of $\d$-exact positive currents which helps us to distinguish two classes among the known non-k\"ahlerian surfaces, namely parabolic and hyperbolic surfaces, see Section \ref{sec:preparations}. 
Hoping that this could lead to a fruitful  approach towards the Global Spherical Shell Conjecture, we formulate in \cite{ChiTo2} three conjectures related to the presence of certain $\d$-exact positive currents on the surface under study. Under the perspective of these conjectures we analyse in the present paper further properties of exact positive $(1,1)$-currents on  compact non-k\"ahlerian surfaces.
More precisely, we discuss the following:
\begin{enumerate}
\item If $T$ is an exact positive $(1,1)$-current on a compact non-k\"ahlerian surface $X$, we will  examine the set $L^{2}_{-1}(T|X)$ of points $x\in X$ around which $T$ is locally in $L^{2}_{-1}$.  We show that on parabolic surfaces $L^{2}_{-1}(T|X)$ contains the complement of the union of all divisors of $X$. We then refine our previous conjecture \ref{conj:Lelong}.
\item We establish the existence of a distingushed current related to the Green function on hyperbolic surfaces.
We relate this to our conjecture \ref{conj:I=0}.
\item 
Related to the conjecture \ref{conj:I>0} we show that any exact positive $(1,1)$-current $T\in L^{2}_{-1}(X)$ with $I(T)\ne0$ induces a singular Gauduchon current on $X$ of a special form, see Section \ref{sec:preparations} for notations and terminology.
\end{enumerate}

\section{Background}\label{sec:preparations}

We refer the reader to the monograph \cite{BHPV} for general facts on compact complex surfaces, to the survey paper \cite{NakamuraSugaku} on more specific facts on non-k\"ahlerian compact complex surfaces and to our paper  \cite{ChiTo2} for more recent results on positive currents on such surfaces.

Let $X$ be non-k\"ahlerian compact complex surface. The existence of non-trivial exact positive $(1,1)$-currents on $X$ implies that the natural map  $j:H^{1,1}_{BC}(X,{\mathbb R})\to H^{1,1}_{A}(X,{\mathbb R})$ from  the Bott-Chern cohomology of $X$ in degree $(1,1)$ to the corresponding Aeppli cohomology has a non-trivial kernel. It can be moreover shown that this kernel is one-dimensional, cf.  \cite[Appendix]{ChiTo2}. We denote by $\tau$ a smooth representative of a class of a non-trivial exact positive $(1,1)$-current in the Bott-Chern cohomology group $H^{1,1}_{BC}(X,\R)$. Its class may be thought of as a positive generator of $\Ker(H^{1,1}_{BC}(X,{\mathbb R})\to H^{1,1}_{A}(X,{\mathbb R}))$.

We denote by $L^2_{-1}(X)$ the space of $(1,1)$-currents on $X$ with coefficients in $L^2_{-1}$. 
The degeneration of the Fr\"olicher spectral sequence at the $E_{1}$-level for compact complex surfaces  implies that any exact  $(1,1)$ current $T$ on $X$ admits a "primitive current", that is  a bidegree $(0,1)$-current $S$ 
such that $T=\partial S$ and $\bar\partial S=0$. It is  not difficult to see that in this case we also have  $\bar\partial S=0$, hence  $T=\d S$. 

In \cite{ChiTo2} we show that if $T$  is positive of type $(1,1)$,  in $L^{2}_{-1}(X)$, $\d$-exact and if $S$ is a primitive current of $T$ as above, then 
$S$ is in $L^2(X)$ and $i\bar S\wedge S$ is $i\partial\bar\partial$-closed 
hence defines 
 a linear form 
$$Ker(H^{1,1}_{BC}(X,{\mathbb R})\to H^{1,1}_{A}(X,{\mathbb R}))\to \R, \ \ \alpha\mapsto \int_{X}\alpha\wedge i\bar S\wedge S.$$ 
This linear form only depends on $T$ and not on the chosen primitive current $S$. We will denote it by $I(T)$. 
When $T$ is moreover smooth the vanishing or non-vanishing of the form $I(T)$ served in \cite{ChiTo} to distinguish between two fundamentally different dynamical behaviours of the induced holomorphic foliation on $X$. This distinction is investigated also for $T$ singular in the case when  $X$ is a known non-k\"ahlerian compact complex surface and fits to the following clssification proposed in \cite{ChiTo2}.

\begin{Def} (\cite[Definition 3.5]{ChiTo2})
We say that a non-k\"ahlerian compact complex surface $X$ is {\em parabolic} 
 if its minimal model belongs to one of the classes: Hopf surfaces, Enoki surfaces, non-k\"ahlerian elliptic surfaces. 
We say that $X$  is  {\em hyperbolic} if its minimal model is either an Inoue surface, an Inoue-Hirzebruch surface or an intermediate Kato surface.
 \end{Def}

We may express the results in \cite{ChiTo2} concerning the vanishing of the forms $I(T)$ as follows.

  \begin{Prop}\label{prop:distinction}
  Let $X$ be a non-k\"ahlerian compact complex surface. 
\begin{enumerate}
\item If $X$ is hyperbolic then all its exact positive $(1,1)$-currents $T$ are in $L^2_{-1}(X)$ and have $I(T)=0$.
\item If $X$ is parabolic then $X$ admits exact positive $(1,1)$-currents $T$ not in   $L^2_{-1}(X)$. Morever   all its exact positive $(1,1)$-currents $T$ in   $L^2_{-1}(X)$ necessarily have $I(T)\ne0$.
\end{enumerate}
\end{Prop}

We recall the following conjectures from \cite{ChiTo2} which together give a partial converse to Proposition \ref{prop:distinction}.

\begin{Conj}\label{conj:I=0}
If $X$ is a non-k\"ahlerian surface all of whose exact positive $(1,1)$-currents $T$ are in $L^2_{-1}(X)$ and satisfy $I(T)=0$, then $X$ 
 is hyperbolic.
\end{Conj}

\begin{Conj}\label{conj:I>0}
If $X$ is a non-k\"ahlerian surface all of whose exact positive $(1,1)$-currents $T$ are in $L^2_{-1}(X)$ but do not all satisfy $I(T)=0$,   then $X$ admits a cycle of rational curves.
\end{Conj}

\begin{Conj}\label{conj:Lelong}
If $X$ is a non-k\"ahlerian surface admitting an exact positive $(1,1)$-current not in $L^2_{-1}(X)$, then there exists on $X$ some exact positive current $T$ with a non-vanishing Lelong number at at least one point of $X$.
\end{Conj}
 

\section{The $L^2_{-1}$-locus}\label{section:currents}

In this section we state a refinement of  Conjecture \ref{conj:Lelong}  and show that it is verified for all parabolic surfaces.

\begin{Def} For a $(1,1)$-current $T$  on a complex surface $X$ we define its {\em  $ L^{2}_{-1}$-locus, }  $ L^{2}_{-1}(T|X)$,  {\em on $X$} as the set of points on $X$ around which $T$ is locally in  $ L^{2}_{-1}$.
 \end{Def}

\begin{Conj}\label{conj:Lelong refined}
If $X$ is a non-k\"ahlerian surface admitting an exact positive $(1,1)$-current $T$ not in $L^2_{-1}(X)$ and if $x\in X\setminus L^{2}_{-1}(T|X)$,  then there exists on $X$ some exact positive $(1,1)$-current  with non-vanishing Lelong number at $x$.
\end{Conj}

Note that by \cite[Proposition 3.6]{ChiTo2} all exact positive $(1,1)$-currents on a hyperbolic surface $X$ are in $L^2_{-1}(X)$, so the above conjecture automatically holds for such surfaces.
 
 As remarked in \cite[end of Section 4]{ChiTo2} using a result of \cite{TelemanFamilies} we have the following. 
 \begin{Rem}
 Conjecture \ref{conj:Lelong refined} may be reformulated as follows.
 ``If $T$ is a non-trivial exact positive $(1,1)$ current on a non-k\"ahlerian surface $X$ and $x$ is a point in $X\setminus L^{2}_{-1}(T|X)$, then there exists an effective divisor $D$ on $X$ homologuous to zero on $X$ and such that $x\in \Supp(D)$.''
 \end{Rem}
 
This will be a consequence of the following Propositions.

\begin{Prop}\label{prop:Enoki}
All   non-trivial
 exact positive $(1,1)$-currents on an Enoki surface $X$ are positive multiples of the current of integration along the cycle of rational curves on $X$.
\end{Prop} 

\begin{Prop}\label{prop:elliptic}
Let $X$ be a non-k\"ahlerian elliptic surface and $F$ be a fiber of its elliptic fibration. Then there exists some exact positive $(1,1)$-current $T$ on $X$ with  $L^{2}_{-1}(T|X)=X\setminus F$ and with vanishing Lelong numbers at all points of $X$. On the other hand the current of integration along $F$ is exact. 
\end{Prop} 

\begin{Prop}\label{prop:Hopf}
Let $X$ be a non-elliptic Hopf surface.  
\begin{enumerate}
\item If $X$ is of class $0$, then all its non-trivial
 exact positive $(1,1)$-currents  are positive multiples of the current of integration along the unique elliptic curve of $X$.
 \item If $X$ is of class $1$ and $C_1$, $C_2$ are its elliptic curves, then any exact positive $(1,1)$-current $T$ on $X$ has local $\partial\bar\partial$-potentials in $L^{\infty}_{1}$ over $X\setminus (C_1\cup C_2)$  and in particular  one has  $L^{2}_{-1}(T|X) \supset X\setminus (C_1\cup C_2)$. Moreover there exist such currents $T$ with  $L^{2}_{-1}(T|X) = X\setminus (C_1\cup C_2)$ and with vanishing Lelong numbers at all points of $X$. On the other hand the current of integration along $C_1\cup C_2$ is exact. 
\end{enumerate}
\end{Prop}

\begin{Rem} 
The above Propositions immediately imply that Conjecture \ref{conj:Lelong refined} holds for minimal parabolic surfaces. Its extension to the non-minimal case needs a bit of care since the condition $T\in L^2_{-1}(X)$ does not behave well under blow-up. Namely, if $p:Y\to {\mathbb B}$ is the blow-up of the origin in the unit ball ${\mathbb B}$ in ${\mathbb C}^2$ and if $T=i\partial\bar\partial\varphi$ is a closed 
positive current on the unit ball in ${\mathbb C}^2$ such that $\partial\varphi\in L^2(\mathbb B)$, then it is not true in general that $\partial p^*\varphi\in L^2(Y)$. 
 See \cite {CoGuZe}, Section 4.2 for a detailed discussion.
 
However on parabolic surfaces the only interesting case where blowing-up may occur in a point away from null-homologuous divisors is the case of Hopf surfaces of class $1$ and in that case we show that exact positive $(1,1)$-currents $T$  have local $\partial\bar\partial$-potentials in $L^{\infty}_{1}$, a property which continues to hold for the pulled-back potentials after blowing up.
These will be therefore locally in $L^{2}_{1}$. 
\end{Rem}

We now give the arguments for the above statements.

Proposition \ref{prop:Enoki} is exactly \cite[Theorem 10 (b)]{TomaCurrents}. 

\begin{proof}[Proof of Proposition \ref{prop:elliptic}]
Let $X$ be a non-k\"ahlerian minimal elliptic surface and let $\pi:X\to Y$ be its elliptic fibration. Here $Y$ is a smooth curve and $\pi$ is known to be a quasi-bundle, which means that all its smooth fibers are isomorphic to one another and its singular fibers are multiples of smooth elliptic curves, see  \cite[Proposition 3.17]{Brinzanescu} for a proof. 
It is immediate to see that the integration current along any fiber of $\pi$ is exact since the De Rham cohomology class of such a fibre vanishes. 

Let $F$ be a fibre of $X\to Y$. We will now show that  positive exact $(1,1)$-currents exist on $X$ which are not  in $L^2_{-1}$ along $F$, are smooth on $X\setminus F$ and have vanishing Lelong numbers at all points of $X$. The idea is to adapt an example of Kiselman \cite[Example 3.1]{Kiselman} to our situation. We consider the function $$z\mapsto u(z):=-(-\log |z|^2)^{\frac{1}{2}}$$
on a small disc $D$ centred at the origin of $\C$ and compose it with $\pi$, where we suppose that the disc $D$ is embedded as an open subset of $Y$, that our fibre $F$ lies over $0\in D$ and that in suitable local coordinates $\pi$ is given by $(w,z)\mapsto z^n$, for some $n\in\N\setminus\{0\}$ at points over $0$, cf.  \cite[p.207]{BHPV} with $n>1$ when the fibre $F$ is singular. There will be no restriction of generality to consider only the case of a smooth fiber which is what we will do for simplicity. We write the function $u$ as $u(z)=\phi(\log|z|^2)$, where $\phi(t):= -(-t)^{\frac{1}{2}}$. Then $\phi$ is increasing and convex on $\R_{<0}$, so $u$  is subharmonic on $D$. The current
$$i\partial\bar\partial u= \frac{i\d z\wedge\d \bar z}{4 |z|^2 (-\log |z|^2)^{\frac{3}{2}}}$$
is positive on $D$, smooth on $D\setminus\{0\}$ and not in $L^2_{-1}(D)$. We may extend it as a (closed) $(1,1)$-current on $Y$ which remains  positive and smooth except at $0$. Then its pullback $T$ to $X$ has the desired properties. Indeed, direct computation shows that with respect to the coordinate systems $(w,z)$ the Kiselman numbers $\nu(T,x,(1,1))$ vanish for all $x\in F$ and these in turn are equal to the classical Lelong numbers $\nu(T,x)$, \cite[III, Corollary 7.3]{DemaillyBook}. 
\end{proof}

The first part of Proposition \ref{prop:Hopf} is contained in \cite[Theorem 69]{HarveyLawson}. Before we  prove the second part we first recall some facts on Hopf surfaces of class $1$ from \cite{KodairaStructureII} and from \cite{HarveyLawson}, see also \cite[Section 3.1.2]{ChiTo2}. 

A  {\em Hopf surface} is a compact complex surface whose universal covering space is isomorphic to 
$\C^2\setminus \{0\}$. It is called {\em primary} if its fundamental group is infinite cyclic, and {\em secondary} otherwise.  Every secondary Hopf surface admits a finite unramified cover which is a primary Hopf surface. The  fundamental group of a  primary Hopf surface is generated by a {\em contraction} $g:\C^2\setminus \{0\}\to\C^2\setminus \{0\}$ which for suitable global holomorphic coordinates $(z_1,z_2)$ on $\C^2$ has the following  form
\begin{equation}\label{eq:contraction2}
g(z_1,z_2)=(\alpha_1z_1+\lambda z_2^m, \alpha_2 z_2),
\end{equation}
where $m\in\Z_{>0}$, $\alpha_1,\alpha_2,\lambda\in\C$ and 
$$(\alpha_1-\alpha_2^m)\lambda=0, \ 0<|\alpha_1|\le|\alpha_2|<1.$$
 A primary Hopf surface $(\C^2\setminus \{0\})/\langle g\rangle$ with $g$ as above is elliptic if and only if $\lambda=0$ and $\alpha_1^{k_1}=\alpha_2^{k_2}$ for some positive integers $k_1$, $k_2$. 
 A non-elliptic Hopf surface $X$ is  {\em of class $1$} if for its primary Hopf cover the coefficient $\lambda$ in \eqref{eq:contraction2} vanishes. In this case 
 its fundamental group is isomorphic to $\Z\times(\Z/l\Z)$ where the direct factor $\Z$ is generated by a contraction $g$ of the form  \eqref{eq:contraction2} and the finite cyclic group $\Z/l\Z$ is generated by an automorphism of $\C^2\setminus \{0\}$ of the form
$$(z_1,z_2)\mapsto(\epsilon_1 z_1, \epsilon_2 z_2),$$
where $\epsilon_1$, $\epsilon_2$ are primitive $l$-th roots of unity. 
Thus $X$ admits an unramified cyclic covering of degree $l$ by the primary Hopf surface $(\C^2\setminus \{0\})/\langle g\rangle$.  A Hopf surface of class $1$ has exactly two irreducible compact curves $E_1$, $E_2$ which are the images by the universal covering map of the  coordinate axes $\{ z_2=0\}$ and $\{ z_1=0\}$ in $\C^2\setminus \{0\}$. The curves $E_1$, $E_2$ are elliptic.

We now restrict our attention to the  primary non-elliptic  Hopf surfaces of class $1$. It will be clear that the properties of exact positive $(1,1)$-currents on secondary Hopf surfaces will be induced by those on their primary Hopf covers. 
So let $X$ be a primary non-elliptic Hopf surface of class $1$ given by a contraction $g$ of the form
$$g(z_1,z_2)=(\alpha_1z_1, \alpha_2 z_2),$$
with $0<|\alpha_1|\le|\alpha_2|<1$ and such that $\alpha_1^{k_1}\ne\alpha_2^{k_2}$ $\forall (k_1,k_2)\in (\N\setminus \{0\})^2$. 
Following \cite{HarveyLawson} we set 
$$r=\frac{\log|\alpha_1|}{\log|\alpha_2|}, \ r'=\frac{1}{r},$$  
$$\phi, \phi': \C^2\setminus \{0\}\to \R, \ \phi(z_1,z_2):=\log(|z_1|^2+|z_2|^{2r}),  \ \phi'(z_1,z_2):=\log(|z_2|^2+|z_1|^{2r'}),$$
$$\eta:=z_2\d z_1-rz_1\d z_2, \ \eta':=z_1\d z_2-r'z_2\d z_1=-r'\eta,$$
$$\Omega:=i\partial\bar\partial\phi=
\frac{|z_2|^{2(r-1)}}{(|z_1|^2+|z_2|^{2r})^2}i\eta\wedge\bar\eta, \ \Omega':=i\partial\bar\partial\phi'=
\frac{|z_1|^{2(r'-1)}}{(|z_2|^2+|z_1|^{2r'})^2}i\eta'\wedge\bar\eta',$$
$$V:=rz_1\frac{\partial}{\partial z_1}-
z_2\frac{\partial}{\partial z_2},$$
$$\pi, \pi':X \to [0,1], \ 
\pi(z_1,z_2):=\frac{|z_1|^2}{|z_1|^2+|z_2|^{2r}},  \ 
\pi'(z_1,z_2):=\frac{|z_2|^2}{|z_2|^2+|z_1|^{2r'}},$$ 
$$\tilde \Omega:=(\psi\circ\pi)\Omega+(\psi\circ\pi')\Omega',$$
where $\psi:[0,1]\to[0,1]$ is smooth and equals $1$ in a neighbourhood of $0$ and $0$ in a neighbourhood of $1$. The forms $\Omega$, $\Omega'$ might not be smooth along $E_1$ and respectively along $E_2$ but $ \tilde \Omega$ is a smooth positive $\d$-closed $(1,1)$-form on $X$ without zeroes on $X$. 
Moreover the holomorphic vector field $V$ defines a holomorphic foliation $\cF$ on $X$, which coincides with the complex foliation defined by $\tilde \Omega$ ( whose leaves are by definition tangent to $\ker(\tilde\Omega)$). We further introduce coordinates $(\rho,\theta_1, \theta_2, \theta_3) $ on $\C^*\times\C^*$ by setting
$$z_1=\rho e^{r\theta_3+i\theta_1}, \ z_2= e^{\theta_3+i\theta_2}$$
with respect to which we get that
$$g(\rho,\theta_1, \theta_2, \theta_3) =(\rho,\theta_1+\Arg(\alpha_1), \theta_2+\Arg(\alpha_2), \theta_3+\log|\alpha_2|),$$
the leaves of the foliation $\cF$ are given by the equations
$$\rho=\mbox{constant}, \ \theta_1-r\theta_2=\mbox{constant},$$
$$\pi([\rho,\theta_1, \theta_2, \theta_3]) =\frac{\rho^2}{\rho^2+1},$$ and its fibers are the $3$-dimensional real tori $\T=\R^3/\langle\tau_1,\tau_2,\varphi\rangle$, where $\langle\tau_1,\tau_2,\varphi\rangle$ is the translation group generated by $\tau_1,\tau_2,\varphi$ over $\Z$ and $\tau_1(\theta_1, \theta_2, \theta_3)=(\theta_1+2\pi, \theta_2, \theta_3)$, 
$\tau_2(\theta_1, \theta_2, \theta_3)=(\theta_1, \theta_2+2\pi, \theta_3)$, 
$\displaystyle{\tau_3(\theta_1, \theta_2, \theta_3)=(\theta_1, \theta_2, \theta_3+2\pi)}$, $\displaystyle{\varphi=\tau_3+\frac{\Arg(\alpha_1)}{2\pi}\tau_1+\frac{\Arg(\alpha_2)}{2\pi}\tau_2}.$ The irrationality assumption $\alpha_1^{k_1}\ne\alpha_2^{k_2}$ $\forall (k_1,k_2)\in (\N\setminus \{0\})^2$ on $(\alpha_1,\alpha_2)$ implies that the leaves of $\cF$ are dense in these tori. We also see that by means of the coordinate functions $(\rho,\theta_1, \theta_2, \theta_3)$ we get a diffeomorphism $X\setminus(E_1\cup E_2)\to\R_{>0}\times\T$.

We next state and prove a lemma which is essentially a rephrasing of \cite[Theorem 58]{HarveyLawson}.

\begin{Lemma}\label{lem:Hopf}
Let $X$ be a non-elliptic Hopf surface of class $1$ and let  $T$ be  a positive $\d$-closed $(1,1)$-current on $X$. If $T'$ is the residual current of $T$ with respect to the Siu decomposition 
$$T=c_1[E_1]+c_2[E_2]+T',$$
then $T'$ is positive $\d$-exact of type $(1,1)$ and there exists a non-negative generalized function $f$ on $]0,1[$ such that on $X\setminus(E_1\cup E_2)$ one has
$$T'=\pi^*(f)\Omega,$$
where $\pi^*(f)$ is the pull-back of  $f$ through the submersion $\pi|_{X\setminus(E_1\cup E_2)}:X\setminus(E_1\cup E_2)\to]0,1[.$ 
\end{Lemma}
\begin{proof}
Let $T$ and $T'$ as in the statement. Since $H^2_{DR}(X,\R)=0$ it is clear that $T'$ is $\d$-exact. The form $\tilde \Omega $ is $\d$-closed so $T'(\tilde\Omega)=0$ and $T'$ is a positive foliation  current. It follows that $T'$ is of the form $T'=k\tilde\Omega$ where $k$ is a non-negative generalized function on $X$. 

Now the statement to prove being local over $]0,1[$ and since $\d\pi\wedge\tilde\Omega=0$, we may suppose that the support of $T'$ lies over some compact subset $K$ of $]0,1[$. We may also suppose by the construction of $\tilde \Omega$ that over $K$ the forms $\Omega$ and $\tilde \Omega$ coincide. So we are left with the task of checking that $k=\pi^*(f)$ for some non-negative generalized function on $]0,1[$. For the reader's convenience we give here an argument which is only a slight modification of that of \cite[Theorem 58]{HarveyLawson}.

We set with respect to new coordinate functions $y=(u,\alpha,\theta_{2},\theta_{3})$, where $\displaystyle{\zeta=\frac{z_1}{z_2^r}}$, $u=\rho^2=|\zeta|^2$, $\alpha=\theta_1-r\theta_2$, 
$$\Omega_1:=
\frac{i\eta\wedge\bar\eta}{|z_2|^{2(r+1)}}=\d u\wedge\d \alpha=i\d\zeta\wedge\d\bar{\zeta}$$
and $T'=h\Omega_{1}$. Since $\Omega_{1}$ differs from $\tilde \Omega$ by a factor which is a function of $\rho$, it will be enough to show that $h$ is of the form $\pi^*(f)$ for some non-negative generalized function on $]0,1[$. 
Let $\psi_\epsilon$ for $\epsilon>0$ be regularizing kernels on $\R^{4}$ and let $h_\epsilon$ be the regularization of $h$ by means of $\psi_\epsilon$, i.e.
$$h_{\epsilon}(y):=\int_{\R^{4}}\psi_{\epsilon}(x)h(y-x)\d x.$$
It is clear that $h_{\epsilon}$ is invariant with respect to the translation group generated by $\tau_{1},\tau_{2},\phi$ mentioned above and hence descends to $X\setminus(E_1\cup E_2)$.  Moreover the smooth forms $T_{\epsilon}:=h_{\epsilon}\Omega_{1}$ are closed since the differential operator  $h\mapsto \d(h\Omega_{1})=(\d h)\wedge\Omega_{1}$ has constant coefficients with respect to the chosen coordinate functions and thus commutes with regularization. It follows that $h_{\epsilon}$ is constant on the leaves of the foliation $\cF$ defined by $\Omega_{1}$. Since these leaves are dense in the fibers of $\pi$, the continuous functions $h_{\epsilon}$ are constant on these fibers and thus
$\displaystyle{\frac{\partial h_{\epsilon}}{\partial \theta_{j}}=0}$ for $1\le j\le3$. Making $\epsilon$ tend to zero now gives $\displaystyle{\frac{\partial h}{\partial \theta_{j}}=0}$ as well. It follows that $h$ is of the form $\pi^*(f)$ for some non-negative generalized function $f$ on $]0,1[$, cf. \cite[IV.5.Exemple 1]{Schwartz}. 
\end{proof}

\begin{proof}[Proof of Proposition \ref{prop:Hopf}(2)]
Le $T$ be an exact positive $(1,1)$-current on a non-elliptic   Hopf surface $X$ of class $1$ and let  $T'$ be the residual part of its Siu decomposition. We will check that the restriction of $T'$ to $X\setminus(E_1\cup E_2)$, which we denote again by $T'$ has coefficients in $L^\infty_{-1}$. 
For this we use the characterization of such residual currents on $X\setminus(E_1\cup E_2)$ given by Lemma \ref{lem:Hopf}. To simplify computations we make use again of our previous notations $\displaystyle{\zeta=\frac{z_1}{z_2^r}}$, $u=\rho^2=|\zeta|^2$, $\alpha=\theta_1-r\theta_2$ and 
$$\Omega_1:=
\frac{i\eta\wedge\bar\eta}{|z_2|^{2(r+1)}}=\d u\wedge\d \alpha=i\d\zeta\wedge\d\bar{\zeta}.$$
 Then for a non-negative generalized function $f$ on $]0,1[$ we have
 $$T'=f(u)\Omega_1.$$
 We claim that an $L^\infty_1$ function $h_1$ in the variable $u$ exists such that $i\partial\bar\partial h_1= T'$ over $]0,1[$.  Indeed, direct computations give
 $$i\partial\bar\partial h_1= (h_1'+uh_1'')\Omega_1$$
 so setting $h=h_1'$ we are led to solving the equation 
 $$ uh'(u)+h(u)=f(u).$$
 Its general solution is of the form
 $$h(u):=\frac{k_0(u)+C}{u},$$
 where $C$ is some real number and  $\displaystyle{k_0(u):=\int_{u_0}^uf(t)\d t}$ is a fixed primitive of $f$. But $k_0$ is locally bounded so $h$ is in $L^\infty_{loc}(]0,1[)$ showing the first part of  Proposition \ref{prop:Hopf}(2).
 
We now construct an example of a current $T'$ on $X\setminus(E_1\cup E_2)$  as above which extends over the curve $E_1$ as a closed positive current $T$ but such that the trivial extension $T$  no longer has $L^2_{-1}$ coefficients locally at points  of $E_1$. The current $T'$ extends if and only if it is locally of finite mass at points of $E_1$ if and only if the trace measure $T'\wedge\omega$ of $T'$ with respect to any positive $(1,1)$-form on $X\setminus E_2$ extends over $E_1$. 

Set $$\omega:=\frac{i\d z_1\wedge\d\bar{z_1}}{|z_1|^2}+
\frac{i\d z_2\wedge\d\bar{z_2}}{|z_1|^{2r'}}.$$ 
We get 
$$\displaystyle{\omega=\left(\frac{\d u}{u}+2r\d\theta_3\right)\wedge\d\theta_1+\frac{2\d\theta_3\wedge\d\theta_2}{u^{r'}}}$$
hence
$$\frac{\omega\wedge\Omega_1}{\omega^2}= \frac{\left(r^2+\frac{1}{u^{r'}}\right)\d\theta_3\wedge\d\theta_2\wedge\d u\wedge\d\theta_1}{2\frac{\d u}{u^{r'+1}}\wedge\d\theta_1\wedge\d\theta_3\wedge\d\theta_2}$$
and $T'$ has finite mass along $E_1$ if and only if $uf$ has finite mass around $0$ on $[0,1[$. 
Assuming this to be the case we now write down under which conditions $T'$ is locally in $L^2_{-1}$ around points of $E_1$. 
If we write $T'=i\partial\bar\partial h_1$ and $h'=h_1$ as above, then $\d h_1=h\d u$ and its pointwise squared norm is $h^2u\Omega_1\wedge\omega/\omega^2$, so $\d h_1$ will have $L^2$ coefficients if and only if $u^2h^2$ is integrable around $0$. 

If we now take $\displaystyle{a\in\left ]-2,-\frac{3}{2}\right ]} $ and 
$$f:]-1,1[\to \R, \ t\mapsto -(a+1)t^a,$$
we find $h(t)=-t^a$ and the current $T':=f(u)\Omega_1$ has finite mass along $E_1$ but its trivial extension $T$ across $E_1$ has zero Lelong numbers everywhere and is not in $L^2_{-1}$ around any point of $E_1$.
\end{proof}

\section{Special currents in the case $I=0$}

Throughout this section we assume that $X$ is a non-K\"ahler compact complex surface on which all $\d$-exact currents are in $L^2_{-1}(X)$ and satisfy $I(T)=0$. Under these 
assumptions, we show that
 there exists a {\it distinguished} positive $\d$-exact current on $X$.

 Denote by ${\mathcal P}$ the cone of positive $d$-exact $(1,1)$ currents on 
$X$. By the above assumptions, we have ${\mathcal P}\subset L^2_{-1}(X)$ and 
if $T=\partial S\in {\mathcal P}$, then $S$ is a $\bar\partial$-closed $(0,1)$ form 
with $L^2$ coefficients and from the condition $I(T)=0$ it follows that $i\bar S\wedge S$ is $d$-exact.

Set 
 \begin{equation*}
 {\mathcal C}=\{ T\in{\mathcal P}\vert \int_XT\wedge g=\int_X i\bar S\wedge S\wedge g\}
 \end{equation*}
 where $g$ is a fixed Gauduchon metric on $X$. Note that the set ${\mathcal C}$ does not depend on $g$. In fact, $T$ and $i\bar S\wedge S$, being $d$-exact, they are cohomologous to a multiple of $\tau$ in the Bott-Chern group $H^{1,1}_{BC}(X,{\mathbb R})$, hence there are $c,c'>0$ and $\lambda$ and $\lambda'$ quasi-plurisubharmonic functions sunch that $T=c\tau+i\partial\bar\partial\lambda$ and $i\bar S\wedge S=c'\tau+i\partial\bar\partial\lambda'$. Then $T$ is in ${\mathcal C}$ iff $c=c'$.

 \begin{Lemma}
 There exists $C$ a positive constant which depends on $g$ such that
 \begin{equation*}
 \int_XT\wedge g\leq C, \forall\, T\in {\mathcal C}
 \end{equation*}
 \end{Lemma}
 \begin{proof}
 Let $T=\partial S\in {\mathcal P}$. Then, from H\"older's inequality, it follows that
 \begin{equation*}
 \int_XT\wedge g=\int_X\partial S\wedge g=\int_XS\wedge \partial g\leq \vert\vert S\vert\vert_{L^2}\vert\vert\partial g\vert\vert_{L^2}=C_1
 \left (\int_Xi\bar S\wedge S\wedge g\right )^{\frac 12}
 \end{equation*}
 and if $T\in {\mathcal C}$, then $\int_XT\wedge g=\int i\bar S\wedge S\wedge g$, hence $\int_X T\wedge g\leq C_1^2$
 \end{proof}
 
 Let $c=\sup \{\int_XT\wedge g\vert T\in {\mathcal C}\}<\infty$.
 
 \begin{Prop}
 There exists a unique $T\in {\mathcal C}$ such that $\int_XT\wedge g=c$.
 \end{Prop}
 \begin{proof}
 Let $T_n\in {\mathcal C}$ such that 
$\int_XT_n\wedge g\to c$. Let 
$T_n=\partial S_n$, where 
$S_n\in L^2_{0,1}(X)$. We can assume that 
$T_n\to T$ weakly (as distributions), and that 
$S_n\to S\in L^2_{0,1}(X)$ weakly (in the Hilbert space 
$L^2_{0,1}(X)$). Then $\int_XT\wedge g=c$, and we want to show that 
$\int_Xi\bar S\wedge S\wedge g=c$. Since $S_n\to S$ weakly in $L^2_{0,1}(X)$, it follows that
 \begin{equation*}
 \int_X i\bar S\wedge S\wedge g= \vert\vert S\vert\vert^2_{L^2}\leq \liminf \vert\vert S_n\vert\vert^2_{L^2}=c
 \end{equation*}
 by the weakly lower semicontinuity of the norm. 
 Set $\displaystyle{c'=\int_Xi\bar S\wedge S\wedge g}$, so that the above inequality implies $c'\leq c$. If $c'=0$, then $S=0$, hence $c=0$, contradiction. Therefore $c'\neq 0$, and if we consider the current $\displaystyle{T'=\frac {c}{c'}T}$, then $T'\in {\mathcal C}$ so $\displaystyle{\int_XT'\wedge g=\frac{c^2}{c'}\leq c}$. This implies that $c'\geq c$, hence $T\in {\mathcal C}$.
 
 In order to prove the uniqueness, consider $T_1=\partial S_1, T_2=\partial S_2\in {\mathcal C}$ such that 
\begin{equation*}
\int_XT_1\wedge g=\int_Xi\bar S_1\wedge S_1\wedge g=\int_X T_2\wedge g=\int_Xi\bar S_2\wedge S_2\wedge g=c
\end{equation*}
and we want to prove that $S_1=S_2$.

Let $\displaystyle{S=\frac 12 (S_1+S_2)}$, $T=\partial S$ and
\begin{equation*}
\lambda=\frac{\int_X T\wedge g}{\int_Xi\bar S\wedge S\wedge g}
\end{equation*}
such that $\lambda T\in {\mathcal C}$. Then $\int_X\lambda T\wedge g\leq c$ and this implies 
\begin{equation*}
\int_Xi\bar S\wedge S\wedge g\geq c=\frac 12\int_Xi\bar S_1\wedge S_1\wedge g+\frac 12\bar S_2\wedge S_2\wedge g
\end{equation*}
therefore
\begin{equation*}
\int_Xi(\bar S_1-\bar S_2)\wedge (S_1-S_2)\wedge g\leq 0
\end{equation*}
Hence the $L^2$ norm of $S_1-S_2$ is $0$, so $S_1=S_2$.
 \end{proof}

We expect the current $T$ in ${\mathcal C}$ for which which $\int_XT\wedge g=c$ to have some additional properties, for instance we think that it should satisfy the equation $T=i\bar S\wedge S$.

\section{Gauduchon currents on non-K\"ahler surfaces}

In this section, we prove that, given a $d$-exact positive current $T$ in $L_2^{-1}(X)$ such that $I(T)\neq0$, then there exists a certain pluriharmonic positive current that dominates some Gauduchon metric. One may call such a current a {\em Gauduchon current}.

\begin{Prop}\label{gc}
Let $X$ be a non-K\"ahler surface, $g$ a 
fixed Gauduchon metric on $X$, $T'$ a nef positive pluriharmonic
 current such that $\int_X T'\wedge\tau>0$, and $T$ a nef non-zero $d$-exact 
positive current.  Then there exist $\varepsilon>0$ and
 $\chi\in {\mathcal D}'(X,{\mathbb R})$ a 
real distribution such that $T'+T+i\partial\bar\partial\chi\geq\varepsilon g$

\end{Prop}

\begin{proof}
First note that Lemme 1.4 in \cite{LamariJMPA}, which is stated for a smooth $(1,1)$-form $\theta$, is also valid in the case of a $(1,1)$-current. Hence, according to this Lemma, it is enough to prove that there exists $\varepsilon>0$ such that $$\int_X(T'+T)\wedge h\geq\varepsilon\int_Xg\wedge h$$ for all Gauduchon metrics $h$ on $X$. Suppose by contradiction, that there exists  a sequence $(h_n)_n$ of Gauduchon metrics on $X$ such that $$\int_X(T'+T)\wedge h_n\leq\frac 1n\int_Xg\wedge h_n,\forall n\geq 1$$ We can normalize the metrics $h_n$ such that $\int_Xh_n\wedge g=1$ and we can assume, without loss of generality, that $(h_n)_n$ converges weakly to a  $i\partial\bar\partial$-closed positive current $R$. It follows that $\int_XT'\wedge h_n\to 0$ and $\int_XT\wedge h_n\to 0$. Since $T$ is nef and $d$-closed, it follows that it is $d$-exact (Proposition 2.4 in \cite{ChiTo2}), and hence there exists $c>0$ and $\lambda$ a quasi-plurisubharmonic function on $X$ such that $T=c\tau+i\partial\bar\partial\lambda$. Therefore $$c\int_XR\wedge\tau=\lim_n c\int_Xh_n\wedge\tau=\lim_n\int_X h_n\wedge T=0$$ and this implies that $R$ is $d$-exact (Proposition 2.4 in \cite{ChiTo2}) and there exists $c'>0$ and $\lambda'$ a quasi-plurisubharmonic function on $X$ such that $R=c'\tau+i\partial\bar\partial\lambda'$. But then $$\langle \{T'\}_A, \{R\}_{BC}\rangle= c'\int_XT'\wedge\tau=\lim_n\int_XT'\wedge h_n=0$$ Here we used Lemma 5.3 in \cite{ChiTo2}. This is a contradiction with the condition $\int_XT'\wedge\tau>0$. From Lamari's result, it follows that there exists $\chi$ a real distribution and $\varepsilon>0$ such that $T'+T+i\partial\bar\partial\chi\geq \varepsilon g$.
\end{proof}

Proposition \ref{gc} implies the following corollary which can be thought of as the singular version of formula $(2.1)$ in \cite{ChiTo}.

\begin{Cor}
Let $X$ be a non K\"ahler surface and $g$ a Gauduchon metric on $X$. Let $T=\partial S$ be a $d$-exact positive current such that $T\in L^2_{-1}(X)$, $I(T)\neq 0$ and $T$ is nef. Then there exists $\chi\in {\mathcal D}'(X,{\mathbb R})$ a real distribution and $\varepsilon>0$ such that $T+i\bar S\wedge S+i\partial\bar\partial \chi\geq \varepsilon g$

\end{Cor}

\begin{proof}
Indeed, the current $i\bar S\wedge S$ is $i\partial\bar\partial$-closed and nef (Proposition 2.2 in \cite{ChiTo2}), so we can use Proposition \ref{gc} to prove the existence of $\chi$ and $\varepsilon$.

\end{proof}

\bibliographystyle{amsalpha}

 \hrule \medskip
\par\noindent

\end{document}